\documentclass[11pt, twoside, leqno]{amsart}    
\usepackage{soul}
\usepackage{lipsum}
\usepackage{amsfonts}
\usepackage{graphicx}   
\usepackage{epstopdf}
\usepackage{calligra}
\usepackage{amsfonts,amsmath,amsthm,amssymb}
\usepackage{mathtools}
\usepackage{hyperref}
\usepackage{autonum}
\usepackage{hhline}
\usepackage{array}
\usepackage[dvipsnames]{xcolor}
\usepackage{diagbox}
\usepackage{tcolorbox}
\usepackage{mdframed}
\usepackage{multicol}
\usepackage{graphicx}
\usepackage{subcaption}
\usepackage{upgreek}
\usepackage{moreverb}
\usepackage{bbm}
\usepackage[margin=1.34in]{geometry}
\usepackage{todonotes}
\usepackage{scalerel,amssymb}
\allowdisplaybreaks
\usepackage{mathrsfs}  
\usepackage[normalem]{ulem}
\usepackage{lineno}
\usepackage{todonotes}
\usepackage{tikz}
\usepackage[makeroom]{cancel}
\usepackage{appendix}
\usepackage{pgfplots}
\usetikzlibrary{arrows.meta}
\usepackage[numbers,sort&compress]{natbib}
\definecolor{mygreen}{HTML}{43a047}
\usepackage{subcaption}
\usepackage{doi}
\usepackage{alphalph}
\usepackage{booktabs}
\usepackage{makecell}
\usepackage{adjustbox}
\usepackage{enumitem}
\makeatletter

\@namedef{subjclassname@2020}{\textup{2020} Mathematics Subject Classification}

\makeatother
\usepackage{xcolor,cancel}

\usepackage{array}
\newcolumntype{H}{>{\setbox0=\hbox\bgroup}c<{\egroup}@{}}

\newcommand{\Om}{\Omega}
\newcommand{\D}{\Delta}

\def \psit{\psi_t}
\def \psitt{\psi_{tt}}
\def \psittt{\psi_{ttt}}

\newcommand{\ddt}{\frac{\textup{d}}{\textup{d}t}}

\newcommand{\ds}{\, \textup{d} s }
\newcommand{\dx}{\, \textup{d} x}

\newcommand{\dxs}{\, \textup{d}x\textup{d}s}

\newcommand{\intt}{\int_0^t}

\newcommand{\intO}{\int_{\Omega}}

\newcommand{\R}{\mathbb{R}} 
 
\newcommand{\Honetwo}{{H_\diamondsuit^2(\Omega)}}
\newcommand{\Honethree}{{H_\diamondsuit^3(\Omega)}}

\newtheorem{theorem}{Theorem}
\newtheorem{lemma}{Lemma}
\newtheorem{proposition}{Proposition}

\newtheorem*{assumption*}{Assumptions}
\newtheorem{assumption}{Assumption}
\newtheorem{definition}{Definition}
\newtheorem{remark}{Remark}
\numberwithin{lemma}{section}
\numberwithin{proposition}{section}
\numberwithin{theorem}{section}
\numberwithin{corollary}{section}
\numberwithin{equation}{section}
\makeatletter
\newcommand{\leqnomode}{\tagsleft@true}
\newcommand{\reqnomode}{\tagsleft@false}
\makeatother

\newcommand{\genk}{\mathfrak{K}}
\newcommand\Lconv{\ast}

\definecolor{grey}{rgb}{0.5,0.5,0.5}
 
\usepackage{stackengine,scalerel}

\colorlet{brown}{brown!80!black}

\definecolor{darkgreen}{rgb}{0,0.5,0}

\newcommand{\alignedintertext}[1]{%
	\noalign{%
		\vskip\belowdisplayshortskip
		\vtop{\hsize=\linewidth#1\par
			\expandafter}%
		\expandafter\prevdepth\the\prevdepth
	}%
}

\title[]{Global existence for a fractionally damped nonlinear Jordan--Moore--Gibson--Thompson equation} 
    
\subjclass[2020]{35L75,35B25}
\keywords{Global existence, Jordan--Moore--Gibson--Thompson, fractional models, nonlinear acoustics, nonlocal damping}
     
\author[M. Meliani and B. Said-Houari]{Mostafa Meliani \and Belkacem Said-Houari}

\address{ 
	Department of Evolution Differential Equations\\ 
	Institute of Mathematics of the Academy of Sciences of the Czech Republic  \\ 
	\v Zitn\' a 25, CZ-115 67 Praha 1, Czech Republic}
\email{meliani@math.cas.cz} 

\address{ 
	Department of Mathematics \\ 
	College of Sciences\\
	University of Sharjah
	\\
	P. O. Box: 27272, Sharjah, UAE}
\email{bhouari@sharjah.ac.ae} 
\begin{document}

\maketitle    
\begin{abstract}
    In nonlinear acoustics, higher-order-in-time equations arise when taking into account a class of thermal relaxation laws in the modeling of sound wave propagation. In the literature, these families of equations came to be known as Jordan--Moore--Gibson--Thompson (JMGT) models. In this work, we show the global existence of solutions relying only on minimal assumptions on the nonlocal damping kernel. In particular,  our result covers the until-now open question of global existence of solutions for the fractionally damped JMGT model with quadratic gradient nonlinearity. The factional damping setting forces us to work with non-integrable kernels, which require a tailored approach in the analysis to control.  This approach relies on exploiting the specific nonlinearity structure combined with a weak damping provided by the nonlocality kernel.
\end{abstract}
\section{Introduction}
We consider  the fractionally damped Jordan--Moore--Gibson--Thompson equation (JMGT): 
\begin{equation}\label{eq:first_eq}
\left\{
\begin{aligned}
		&\tau \psittt+ \psitt
	-c^2  \Delta \psi - \tau c^2  \Delta\psit
	- \delta   \textup{D}_t^{1-\alpha} \D (\tau  \psit + \psi)= 2 \sigma \nabla \psi \cdot \nabla \psi_t  + f(x,t),\\
	&\psi(t=0)= \psi_0 ,\qquad \psit(t=0)= \psi_1,\qquad \psitt(t=0)=\psi_2 \\
	&\psi|_{\partial \Omega} = 0.
	\end{aligned}
	\right.
\end{equation}
Here $x\in \Omega$ and $t>0$, where $\Omega\subset\R^d$ with $d\in \{1,2,3\}$, is a bounded $C^{2,1}-$regular domain.  
The parameters $\tau, \delta$ are positive ($>0$), while $\sigma \in \R$. The fractional derivative parameter $\alpha$ belongs to $(0,1)$. 
 Here $\textup{D}_t^{\gamma}$ denotes the Caputo--Djrbashian \cite{Djrbashian:1966,Caputo:1967} 
 derivative of order $\gamma \in (0,1)$, that is 
\begin{equation}
	\textup{D}_t^{\gamma} v = \frac{1}{\Gamma(1-\gamma)} \intt (t-s)^{-\gamma} v_t \ds \quad \text{for }\ v \in W^{1,1}(0,T), \quad t \in (0,T). 
\end{equation} 
The function $f$ is a source term depending on the space and time variables $(x,t)$ but not on the solution $\psi$.

Equation~\eqref{eq:first_eq}  is a special case of the models of nonlinear acoustics derived in \cite{kaltenbacher2022time} (see also; \cite{jordan2014second}) based on incorporating the average of two generalized Cattaneo heat flux law proposed by Compte and Metzler~\cite{compte1997generalized} in the set of governing equations of acoustic propagation. 
The first equation in \eqref{eq:first_eq} can be conveniently rewritten as an integro-differential equation with an appropriate kernel as follows 
\begin{equation}\label{eq:first_eq_genk}
		\tau \psittt+\psitt
-c^2 \Delta \psi - \tau c^2 \Delta\psit
- \delta \genk \Lconv\Delta (\tau\psitt+ \psit)= 2 \sigma \nabla \psi \cdot \nabla \psi_t + f,
\end{equation}
where 
$\Lconv$ denotes the Laplace convolution, and is interpreted as
\begin{align}
	(\mathfrak{K}\Lconv g)(t)=\int_0^t \mathfrak{K} (s)\,g(t-s)\ds, \quad \quad \textrm{for }\genk, g \in L^1(0,t).
\end{align} 
In particular, to obtain the Djrbashian--Caputo fractional derivative of order $\alpha \in (0,1)$, we can set:
\begin{equation}\label{eq:frac_kernel}
\genk: t \mapsto \frac1{\Gamma(\alpha)} t^{\alpha-1}.
\end{equation}

Equation~\eqref{eq:first_eq_genk} is supplemented by zero-Dirchlet boundary conditions and appropriate initial data, discussed in Section~\ref{sec:Main_Thm}.

\subsection{Related results}
The  study of the initial-boundary value problem of~\eqref{eq:first_eq} (alternatively,~\eqref{eq:first_eq_genk})  with zero-Dirichlet data and sufficiently regular initial data has thus far focused on questions of local existence  in a Hilbertian setting and asymptotic behavior (e.g., rate of convergence as $\alpha \to 0^+$, and $\tau \to 0^+$)  in order to connect different models of acoustics. We refer the reader to~\cite{kaltenbacher2022time,meliani2023unified,kaltenbacher2024vanishing} for such studies.  

The general idea for the well-posedness proofs for the nonlinear problem  in the absence of quadratic gradient nonlinearities ($\sigma=0$ in~\eqref{eq:first_eq} or \eqref{eq:first_eq_genk}) is based on considering  regular and small initial data
\[(\psi_0,\psi_1,\psi_2) \in \big(H^2(\Omega)\cap H_0^1(\Omega)\big) \times \big(H^2(\Omega)\cap H_0^1(\Omega)\big) \times H_0^1(\Omega),\]
for which the existence of the solution can be established on a sufficiently small time interval $[0,T]$; see, e.g.,~\cite{kaltenbacher2022time,kaltenbacher2024vanishing}.  In the presence of gradient nonlinearities, more regular initial data are needed. For instance, in \cite[Theorem 3.3]{kaltenbacher2024vanishing} assumes initial data 
\[(\psi_0,\psi_1,\psi_2) \in H^3(\Omega) \times H^3(\Omega)  \times H^2(\Omega) .\]
The authors of the above mentioned works did not address the global existence and long-time behavior of the solution, leaving these questions open for future investigation.
The main challenge in investigating the question of global existence lies in the absence of uniform-in-time bounds within the energy estimates. 
   
Interest in long-term behavior of solutions for the JMGT equation with fractional damping is more recent and was shown only for a linearized  version of  fractional JMGT equation and for a restricted family of kernels in~\cite{meliani2025energy}.  Among the restrictions on the kernels found in \cite{meliani2025energy} is the need for exponentially decaying kernels, or at least $L^1(\R^+)$ kernels  for a weaker result.
Kernel~\eqref{eq:frac_kernel} verifies neither condition,  i.e., it is not exponentially decaying kernels, nor is it in $L^1(\R^+)$.    In~\cite{meliani2024well}, a global extensibility criterion was proposed for a nonlinear fractionally damped JMGT equation with the help of Brezis--Gallou\"et--Wainger inequalities~\cite{brezis1980nonlinear,brezis1980note}. In particular, we showed that as long as 
\[\|\psi_t\|_{H^{d/2}(\Omega)} + \|\psi_t\|_{H^{d/2-1}(\Omega)} < \infty,  \]
with $d$ being the spatial dimension, then a local solution to \eqref{eq:first_eq_genk}  belonging to a suitable solution class can be extended to be global in time. Our result relied on the assumption that the kernel is coercive (see~\cite[Assumption 1]{meliani2024well}) and complemented the result of \cite{nikolic2024infty} which showed that blow-up of the inviscid form of \eqref{eq:first_eq_genk} (i.e., with $\delta=0$) is driven by an inflation of the $L^\infty$-norm of the solution $\|\psi_t\|_{L^\infty(\Omega)}$ (as opposed to a gradient blow-up).

 When the convolution kernel $\genk$ is allowed to be the Dirac delta pulse $\delta_0$, problem \eqref{eq:first_eq_genk} reduces to the initial-boundary value problem for the nonlinear JMGT equation:
\begin{equation}\label{JMGT_Eq}
		\tau \psittt+ \psitt
	-c^2  \Delta \psi - (\delta+\tau c^2)  \Delta\psit - 2 \sigma \nabla \psi \cdot \nabla \psi_t = 0,
\end{equation}
which can be seen as a special case of the thermally relaxed Kuznetsov equation~\cite[Equation (72)]{jordan2014second}, where the coefficient of nonlinearity is set to 1~\cite[p. 2197]{jordan2014second}.

The related relaxed Westervelt equation:
\begin{equation}\label{JMGT_Eq_Wes}
		\tau \psittt+ (1+2k\psit)\psitt
	-c^2  \Delta \psi - (\delta+\tau c^2)  \Delta\psit  = 0,
\end{equation}
has received comparatively higher attention in recent years, as the absence of quadratic gradient nonlinearities makes it easier to study. Global existence and exponential decay for equation~\eqref{JMGT_Eq_Wes} has been established in  \cite{kaltenbacher2012well} for $\delta>0$. A similar study was carried out for its linearized counterpart ($k=0$, $\delta>0$) in~\cite{kaltenbacher2011wellposedness} 
where the authors used the energy method together with a bootstrap argument to show, under a smallness assumption on  the initial data, a global existence result and decay estimates for the solution.  See also the decay results for the linearized problem in \cite{pellicer2019wellposedness} and \cite{Chen_Ikehata_2021}. The interested reader is referred to~\cite{KaltenbacherNikolic,conejero2015chaotic,Klaten_2015,marchand2012abstract,P-SM-2019,said2022global} for various studies related to \eqref{JMGT_Eq}. 

 As for equation \eqref{JMGT_Eq}, whose prominent feature is the quadratic gradient nonlinearity: $2\sigma \nabla \psit\cdot\nabla\psi$, there exist far fewer results. A notable result is the global existence result of solution to the Cauchy problem for the JMGT equation with Kuznetsov nonlinearity:
\begin{equation}\label{eq:ChenKuz}
		\tau \psittt+  (1+2k\psit)\psitt
	-c^2  \Delta \psi - (\delta+\tau c^2)  \Delta\psit - 2 \sigma \nabla \psi \cdot \nabla \psi_t = 0,
\end{equation}
is due to Racke and Said-Houari~\cite{racke2021global}. Naturally, their result also covers \eqref{JMGT_Eq_Wes} by setting $\sigma=0$. 
 We also mention here the work of \cite{chen2022influence} who proved finite-time blow-up of the solution for the inviscid (i.e., $\delta = 0$) JMGT equation~\eqref{eq:ChenKuz}
in $\R^n$ for appropriately chosen initial data.

Our goal in this work is  to show that the fractional damping is strong enough to prevent the formation of this blow-up and that we can show global boundedness of the solution for small enough data.

In this work, we use a similar bootstrapping idea as  \cite{racke2021global} to prove global existence of solutions for small enough data.

\subsection{Novelty}
Our work answers the fundamental question of whether fractional damping is strong enough to produce sufficient smoothness for a solution to exist globally.  Indeed, in equation~\eqref{eq:first_eq_genk}, no damping can be extracted from the term $ -\tau c^2 \D \psit$ as we are in the so-called critical case, thus our arguments will need to make primarily use of the term $- \delta \genk \Lconv \D \psitt$ to prevent blow up of the solution.
Working with nonlocal damping has clear drawbacks from the point of view of analysis. Indeed, a quick inspection of the convolution kernel $\genk$, shows that it is not even $L^1(\R^+)$ making it difficult to control the nonlinear and nonlocal terms without first extracting some control on the norm of the solution by exploiting the kernel term; see Lemma~\ref{lemma_L2_control} below. Thus, our main ingredient will be to exploit the concept of strongly positive kernels introduced by Nohel and Shea~\cite[Corollary 2.2]{nohel1976frequency}; see the characterization of strongly positive kernels given in Lemma~\ref{thm:strong_pos_ker}. Furthermore, we  need to exploit the  structure of the nonlinearity, which can be written as a derivative in time: $2 \sigma \nabla \psi \cdot \nabla \psi_t =\sigma\partial_t(|\nabla\psi|^2)$. 
This will allow us to integrate the equation in time and obtain the necessary estimates to prove the global existence; see Section~\ref{sec:nabla_v_dmp} for more details.

\subsection{Organization of the paper}
The rest of the paper is organized as follows: In Section~\ref{sec:kernel_assu}, we discuss the main assumptions made on the nonlocal damping kernel and relevant properties of convolutions. The main result of the paper (Theorems \ref{thm:main_thm}) is presented in Section~\ref{sec:Main_Thm}. Section~\ref{sec:proof_Main_Theorem} is then devoted to the proof of  Theorem~\ref{thm:main_thm} through energy arguments. Appendix~\ref{App:Appendix} provides a local well-posedness result for the initial-boundary value problem related to \eqref{eq:first_eq_genk} with a source term $f$.

\section{Preliminaries and assumptions} \label{sect2}
In this section, we introduce a few notations, and list some necessary assumptions on the kernel $\genk$. In addition,  we collect some  helpful embedding results and inequalities that we will repeatedly use in the proof of the main results. 

\subsection{Notations}
Throughout the paper,  the letter $C$ denotes a generic positive constant
that does not depend on time, and can have different values on different occasions.  
We often write $f \lesssim g$ where there exists a constant $C>0$, independent of parameters of interest such that $f\leq C g$. 
We often omit the spatial and temporal domain when writing norms; for example, $\|\cdot\|_{L^p L^q}$ denotes the norm in $L^p(0,T; L^q(\Omega))$.  

 We assume throughout that
$\Omega \subset \R^n$, where $n \in \{1, 2, 3\}$, is a bounded $C^{2,1}$-regular domain. This will allow us to use desired elliptic regularity results; see estimate~\eqref{eq:elliptic_reg}. We also introduce, similarly to the work~\cite{kaltenbacher2024kuznetsov} where a similar quadratic nonlinearity was treated, the following functional Sobolev spaces which are of interest in the analysis:
\begin{equation} \label{sobolev_withtraces}
	\begin{aligned}
		\Honetwo:=&\,H_0^1(\Omega)\cap H^2(\Omega), \ 
		\Honethree:=\, \left\{u\in H^3(\Omega)\,:\,  u|_{\partial\Omega} = 0, \  \D u|_{\partial\Omega} = 0\right\}.
	\end{aligned}
\end{equation}

\subsection{Assumption on the kernel $\genk$}\label{sec:kernel_assu}
In order to state and prove our main result, we   make the following assumptions on the kernel $\genk$.

\begin{assumption}\label{assu:kernel}
We assume that $\genk$ is a locally integrable function on $[0,\infty)$. Furthermore, we assume that it is twice differentiable with $\genk_t \neq 0$ satisfying
	$$(-1)^n \genk^{(n)} (t) \geq 0 \qquad \forall t>0,\ n=0,1,2$$
\end{assumption}

	We intend here to show that a kernel $\genk$ verifying Assumption~\ref{assu:kernel} can be used to provide some damping; see Lemma~\ref{lemma_L2_control} below. To this end let us recall a definition and a characterization of strongly positive definite kernels.
	\begin{definition}[Strongly positive kernel \cite{nohel1976frequency}]
		A real function $h \in L^1_\textup{loc}(0,\infty)$ is said to be strongly positive with constant $\eta$ if there exists $\eta>0$ such that the function $g: t \mapsto h(t) - \eta e^{-t}$ is of positive type. That is if 
		\begin{equation}
			\intt  (g \Lconv y )(s) y (s) \ds \geq 0,
		\end{equation}
	or equivalently, if 
	\begin{equation}
	\intt  (h \Lconv y )(s) y (s) \ds \geq \eta \intt  (e^{-t} \Lconv y )(s) y (s) \ds.
	\end{equation}
\end{definition}

Note that $t\mapsto e^{-t}$ is a positive kernel~\cite[pp 279-280]{nohel1976frequency} and thus a strongly positive kernel is naturally also a positive kernel, i.e.,
\begin{equation}\label{ineq:positive_kernel}
    \intt  (h \Lconv y )(s) y (s) \ds \geq 0.
\end{equation}

We recall the result from~\cite[Corollary 2.2]{nohel1976frequency}, which is fundamental in establishing the strong positivity of many kernel classes discussed in the literature.
\begin{lemma}
    [Characterization of strongly positive kernels]\label{thm:strong_pos_ker}
	A twice differentiable function $h$ with $h_t \neq 0$ satisfying
	$$(-1)^n h^{(n)} (t) \geq 0 \qquad \forall t>0,\ n=0,1,2$$
	is a strongly positive kernel.
\end{lemma}

\begin{lemma}[Lemma 2.4 in \cite{kawashima1993global} and Lemma 2.9 in \cite{cannarsa2011integro}.]\label{lemma_L2_control}
	Suppose that $y\in L^2_\textup{loc}(0, \infty)$ and $y_t \in L^2_\textup{loc}(0, \infty)$.
	Let $\genk$ be a strongly positive kernels with constants $\eta$. Then the following holds:
		\begin{equation}\label{ineq:L2_lemma}
			\intt |y(s)|^2\ds \leq |y(0)|^2 + 2 \eta^{-1}  \intt  (\genk \Lconv y )(s) y (s) \ds +  2 \eta^{-1}  \intt  (\genk \Lconv y_t )(s) y_t (s) \ds.
		\end{equation}
\end{lemma}

\subsection{Examples of relevant kernels}\label{sec:coverd_kernels}
Although the Abel kernel~\eqref{eq:frac_kernel} is our primary example of interest, Assumption~\ref{assu:kernel} is sufficiently general to encompass a broad class of kernels arising in the modeling of acoustic and visco-elastic phenomena. We mention here some of these kernels which verify Assumption~\ref{assu:kernel}:

\begin{itemize}
    \item The exponential kernel 
 	\begin{equation}\label{Exponential_Kernel}
\genk(t) = e^{-\beta t}, \qquad \beta>0,
\end{equation}.

\item The exponentially regularized Abel kernel (along the lines of that found in~\cite{messaoudi2007global})
	$$\genk =\frac{t^{\alpha-1}e^{-\beta t}}{\Gamma(\alpha)},$$
	with $0<\alpha<1,$ $\beta>0$.

\item The fractional Mittag-Leffler kernels encountered in the study of fractional second order wave equations in complex media in~\cite{kaltenbacher2024limiting}
$$\genk =\frac{t^{\beta-1}}{\Gamma(1-\alpha)}E_{\alpha,\beta}(-t^\alpha),$$ 
with $0<\alpha\leq\beta\leq1$, where $E_{\alpha,\beta}$ is the two-parametric Mittag-Leffler function. 

\item The polynomially decaying kernel 
	\begin{equation}\label{Polynomial_Kernel}
\genk(t) = \frac{1}{(1+t)^{p}}, \qquad p>1
\end{equation}
arising in viscoelasticity, see, e.g.,~\cite{munoz1996decay}.
\end{itemize}
 
\subsection{Elliptic regularity estimates}
We recall here a useful elliptic regularity result to be used in establishing the necessary energy estimates; see, e.g.,~\eqref{eq:elliptic_reg} below.
\begin{lemma}
    Let $\Omega$ be a bounded $C^{2,1}$-regular domain. Let $\psi$ be a function  such that $\D \psi \in L^2(\Omega)$ and $\psi|_{\partial\Omega} =0$, then 
    \begin{equation}\label{low_elliptic_reg_est}
        \|\psi\|_{H^2(\Omega)} \lesssim \|\D \psi\|_{L^2(\Omega)}.
    \end{equation}

    \noindent If instead $\nabla \D \psi \in L^2(\Omega)$ and $ \psi|_{\partial\Omega} = \D \psi|_{\partial\Omega} =0 $, then 
    \begin{equation}\label{high_elliptic_reg_est}
        \|\psi\|_{H^3(\Omega)} \lesssim \|\nabla \D \psi\|_{L^2(\Omega)}.
    \end{equation}
\end{lemma} 
\begin{proof}
Consider the following elliptic PDE
   \begin{equation}
   \left\{
    \begin{array}{cc}
         -\D \psi= g, & \qquad \textrm{on } \Omega,  \\
         \psi|_{\partial \Omega} = 0 & 
    \end{array}   
    \right.
   \end{equation}
   with $g\in H_0^1(\Omega)$ (since $\D \psi|_{\partial \Omega} =0$). Note that $\|g\|_{H_0^1(\Omega)} = \|\nabla \D \psi\|_{L^2(\Omega}$.  
   Thus, from classic elliptic regularity results~\cite[Theorem 2.5.1.1]{grisvard2011elliptic}, we have the estimate:
   \begin{equation}
   \|\psi \|_{H^3(\Om)} \lesssim \|g\|_{H^1(\Omega)} \sim \|\nabla \D \psi\|_{L^2(\Omega)}.
   \end{equation}
   Estimate~\eqref{low_elliptic_reg_est} is obtained in a similar manner.
\end{proof}

\subsection{A crucial functional inequality}
 We conclude the preliminaries section by introducing an important lemma that we will use to establish a uniform bound of the total energy of  the solution; see Proposition~\ref{lem::Energy_Diss_1D}. This lemma was proved in \cite[Lemma 3.7]{Strauss_1968}.

\begin{lemma}
\label{Lemma_Stauss} Let $M=M(t)$ be a non-negative continuous function
satisfying the inequality
\begin{equation}  
M(t)\leq c_1+c_2 M(t)^{\kappa},
\end{equation}
in some interval containing $0$, where $c_1$ and $c_2$ are positive
constants and $\kappa>1$. If $M(0)\leq c_1$ and
\begin{equation}
c_1c_2^{1/(\kappa-1)}<(1-1/\kappa)\kappa^{-1/(\kappa-1)},
\end{equation}
then in the same interval
\begin{equation}
M(t)<\frac{c_1}{1-1/\kappa}.
\end{equation}
\end{lemma}


\section{Main results}\label{sec:Main_Thm}
 In this section, we present the main results of this paper and explain the strategy of the proof.
For the convenience of the analysis, we rearrange equation \eqref{eq:first_eq_genk} as follows
\begin{subequations}\label{Main_system}
\begin{equation}
		\tau \psittt + \psitt- c^2\Delta \psi-\tau c^2 \Delta \psit- \delta \genk \Lconv\Delta (\tau\psitt+ \psit)= \sigma \nabla \psi \cdot \nabla \psi_t + f ,
    \label{System_New}
\end{equation}
where the source term here $f = \ddt \tilde f\in L^1(\R^+; \Honetwo)$ with $\tilde f \in W^{1,1}(\R^+;\Honetwo$. We assume the following on the  initial data
\begin{equation}  \label{Initial_Condition_2}
	\psi(t=0)= \psi_0 \in \Honethree ,\qquad \psit(t=0)=\psi_1=-\frac{1}{\tau}\psi_0 ,\qquad \psitt(t=0)=\psi_2 \in \Honetwo,
\end{equation}
 with the additional requirement that $\nabla\psi_0| _{\partial{\Omega}}=0$. We furthermore impose the zero-Dirichlet boundary condition:
\begin{equation}  \label{Boundary_Condition_3}
	\psi|_{\partial \Omega} = 0.
\end{equation}
\end{subequations}
\begin{remark}[On the requirements on the data]
The seemingly peculiar form of the requirement on the source term $f$ has to do with our intention to integrate equation \eqref{System_New} in time in Section~\ref{sec:nabla_v_dmp}. We point out also that requiring special form of the initial data is not uncommon when studying fractional/nonlocal-in-time wave equations; e.g.,  \cite{kaltenbacher2022time,meliani2023unified}, or even local-in-time Moore--Gibson--Thompson equations; e.g., \cite{Chen_Ikehata_2021}  where some of the initial data are set to $0$. In our work, the initial data restriction leads to
$$\left(\tau \psit +\psi\right) (t=0) =0$$
which is used when integrating \eqref{System_New} in time.
\end{remark}

Let us introduce the space of solutions used in the statement of Theorem~\ref{thm:main_thm}:
\begin{equation}\label{def:space_of_solutions}
	\begin{aligned}
	\mathcal{U}(0,T) = \Bigg\{\psi \in W^{3,1}(0,T; H_0^1(\Omega)) \cap  W^{2,\infty}(0,T;\Honetwo)& \cap   W^{1,\infty}(0,T;\Honethree)\Bigg| \\ \sigma \nabla\psi\cdot\nabla\psit \in H_0^1(\Omega)\Bigg\}.
	\end{aligned}
\end{equation}

\begin{theorem}[Global existence for small data]\label{thm:main_thm}
Let $\Omega$ be a $C^{2,1}$ domain. Let Assumption~\ref{assu:kernel} on the kernel hold. Then, there exists a small $m>0$ such that if the initial data and source term verify  

\begin{equation}
    \|\psi_0\|_{H^3(\Om)}+\|\psi_2\|_{H^2(\Om)}+\|\tilde{f}\|_{W^{1,1}(\R^+;H^2(\Om)(\Omega))} \leq m,
\end{equation}
    Then the local solution $\psi\in \mathcal{U}$ to \eqref{System_New}--\eqref{Boundary_Condition_3} exists globally in time.
\end{theorem}

\subsection{Discussion of the main result}
 Before moving onto the proof, we briefly discuss the statements made above in Theorems~\ref{thm:main_thm}. 
 \begin{itemize}
 \item Theorem~\ref{thm:main_thm} proves the existence of global-in-time solution to \eqref{System_New}--\eqref{Boundary_Condition_3} for small enough data. The main difficulty in showing a global existence result lies in proving that the nonlocal term $\delta \genk\Lconv\D(\tau \psitt +\psit)$  provides sufficient control to ensure that the solution does not blow-up.  We recall a notable existing result in the direction of global existence for the local JMGT equation with $\delta =0$ in the case of Westervelt-type nonlinearit (see \eqref{JMGT_Eq_Wes}) is a conditional regularity result due to Nikoli\'c and Winkler~\cite{nikolic2024infty}.   For the nonlocal JMGT equation with a general convolution kernels $\genk$ and Westervelt nonlinearity, a conditional regularity result was established recently in \cite{meliani2024well}. On the other hand, long-term behaviour of nonlocal equations is, in general, poorly understood~\cite{fritz2022equivalence}. The combination of these facts, makes the study of \eqref{System_New}--\eqref{Boundary_Condition_3} challenging and interesting.
  
 \item  Showing the large-time asymptotic behavior solutions is an interesting problem which strongly depends on the type of the kernel $\genk$. For exponentially decaying kernel, it is possible to prove an exponential decay of the solution in particular cases which do not cover \eqref{eq:first_eq_genk}; see, e.g., \cite{meliani2025energy,messaoudi2007global}. However, proving decay rates for more general classes of kernels is more difficult. This is the case only in particular instances of lower-order equation with a simpler structure compared to \eqref{eq:first_eq_genk}; see, e.g.,~\cite{vergara2008lyapunov,vergara2015optimal,fritz2022equivalence}. Studying the asymptotic behavior of nonlinear equations with higher-order derivatives is a challenging task which remains open for the fractional JMGT equation at hand.
 \end{itemize}

The proof of Theorem~\ref{thm:main_thm} is given in  Section \ref{sec:proof_Main_Theorem}.

\section{Proof of Theorem~\ref{thm:main_thm}}\label{sec:proof_Main_Theorem}
The proof of Theorem \ref{thm:main_thm} will be done though a bootstrap argument. We define first the energy associated to \eqref{Main_system} as 
\begin{equation}\label{Total_Energy}
    \mathbf{E}(t):=\sup_{\nu \in (0,t)} E_1(\nu) + E_2(\nu)
\end{equation}
where  
\begin{equation}  
	E_{1}(t):=\frac{1}{2}\int_{\Omega}\left( |\D(\psit+\tau \psitt)|^{2}+c^2|\nabla \D  (\psi+\tau \psit)|^{2}\right) \dx
\end{equation}
and 
\begin{equation}
    E_2(t):=\frac{1}{2}\int_{\Omega}\left( |\D(\psi+\tau \psit)|^{2}+c^2|\nabla \D  (\xi+\tau \psi)|^{2}\right) \dx,
\end{equation}
where $\xi = \intt \psi \ds + \xi_0$ with $\xi_0$ defined through the elliptic problem~\eqref{inital_cond_xi_0} below.

We define the  associated dissipation rate 
\begin{equation}\label{D_Term}
    \mathbf{D}(t):= \frac{\delta }{\eta} \intt \|\nabla \D (\tau  \psit +  \psi) \|_{L^2(\Omega)}^2. 
\end{equation}
We introduce  
\begin{equation}
    \mathbf{Y}(t):= \mathbf{E}(t) + \mathbf{D}(t).  
\end{equation}
Our main goal is to prove by a continuity argument that   $\mathbf{Y}(t)$ is  uniformly bounded for all time provided that  the initial energy $\mathbf{E}(0)=\mathbf{Y}(0)$  and source term $\|\tilde f\|_{W^{1,1}(\R^+,L^2(\Omega))}$ are sufficiently small.
Such a functional will be shown to verify the  inequality 
\begin{equation}\label{Main_Y_Estimate}
\mathbf{Y}(t)\lesssim  \mathbf{Y}(0) + \|\tilde f\|_{W^{1,1}(\R^+; H^2(\Omega))}+\mathbf{Y}^{2}(t).
\end{equation}
 where the hidden constant in \eqref{Main_Y_Estimate} is independent of $t$.
Equipped with the above inequality, we can show by using Lemma \ref{Lemma_Stauss} that  the energy of the solution is uniformly bounded as time goes  to infinity. Hence, our primary goal is to prove \eqref{Main_Y_Estimate}. To do so, it is natural to compute the time evolution of $\mathbf{E}(t)$.

Our starting point in the analysis is  the following energy estimate:

\begin{lemma}\label{lemma_E1_identitty}
	We have  for all $t\geq 0$, the identity
	\begin{equation}
		E_{1}(t) + \delta  \int_0^t   \int_{\Omega} \genk\Lconv \nabla \D(\tau  \psitt +\psit) \cdot  \nabla \D  (\tau \psitt + \psit)\dx\ds \leq E_1(0)+\mathrm{R}_1,  \label{Energy_Indentity}
	\end{equation}
	where
	\begin{equation}
		\mathrm{R}_1:=\int_0^t\int_{\Omega}2\sigma \D (\nabla \psi \cdot \nabla \psit) \cdot 
		\D \left(\psit+\tau  \psitt\right) \dx\ds   + \int_0^t\int_{\Omega} \D f\, \cdot \D \left( \psit+\tau \psitt\right) \dx\ds.
	\end{equation} 
\end{lemma}
\begin{proof}
     We intend to test \eqref{System_New} by $ \D^2 (\psit+\tau \psitt)$ which lacks the sufficient regularity to be a valid test function. Therefore, similarly to the idea of, e.g., \cite[Theorem 8.3]{conti2021Moore} and \cite[Lemma 5.2]{meliani2025energy}, we will carry out the testing in a semi-discrete setting (Galerkin approximation), such that the differential inequality~\eqref{Energy_Indentity} will make sense at first for the Galerkin approximations of the equation. Taking the limit in the Galerkin procedure will then ensure that the inequality also holds for the solution of the infinite dimensional problem.

	Thus, multiplying \eqref{System_New} by $ \D^2 (\psit+\tau \psitt)$ and integrating by parts
	over $\Omega$, we obtain
	\begin{equation} 
    \begin{multlined}
		\frac{1}{2}\ddt\int_{\Omega}|\D (\psit+\tau \psitt)|^{2} \dx+  \frac{1}{2}\ddt\int_{\Omega}c^2|\nabla \D  (\psi+\tau \psit)|^{2} \dx \\ + \delta  \int_{\Omega} \genk\Lconv \nabla \D(\tau  \psitt +\psit) \cdot  \nabla \D (\tau \psitt + \psit)\dx \notag \\
		=  \int_{\Omega}2\sigma (\nabla \psi \cdot \nabla \psit)  
		\D^2 \left(  \psit+\tau  \psitt\right) \dx + \int_{\Omega} f \D^2 \left( \psit+\tau \psitt\right) \dx.  \label{w_estimate}
        \end{multlined}
	\end{equation}%
    We next integrate by parts the right-hand side using the fact that $f \in \Honetwo$ and thus $f|_{\partial \Omega} = 0$ and  recalling also through the PDE \eqref{Main_system} and the function space $\mathcal{U}$ (see \eqref{def:space_of_solutions}), that 
    $2 \sigma \nabla \psi\cdot \nabla \psit |_{\partial \Omega} = \D (\tau \psitt +\psit)|_{\partial \Omega}= 0$, which yields after  integrating over   $(0,t)$  the desired estimate.
\end{proof}

\subsection{Control of $\|\nabla \psit\|_{L^2(0,T;L^2(\Omega))}$}\label{sec:nabla_v_dmp}
    One of the main challenges in closing the estimates for the nonlinear equation is controlling the norm  $\|\nabla \psit\|_{L^2(0,T;L^2(\Omega))}$; see proof of Proposition~\ref{lem::Energy_Diss_1D}. Since, we are in the critical case (in terms of the coefficient \cite{kaltenbacher2012well}), we need to use crucially the properties of the kernel. In particular, we need to  rely on Lemma~\ref{lemma_L2_control} together with an approach inspired by \cite[Section 1.9]{lions1965some}.
	 The idea is to exploit the structure of the nonlinearity which can be written as a derivative in time: $2\sigma \nabla \psi \cdot \nabla\psit = \sigma(|\nabla\psi|^2)_t$. Then, by averaging in time, we can use \eqref{ineq:L2_lemma} to derive \eqref{full_energy_ineq} through appropriate energy estimates.

	Define 
\begin{equation}\label{xi_Definition}
\xi = \xi_0 + \intt \psi(s)\ds, 
\end{equation}
	with $\xi_0$ being the solution to the following elliptic problem on $\Omega$:
       \begin{equation}\label{inital_cond_xi_0}
   \left\{
    \begin{array}{rlc}
         -c^2\D\xi_0 &= \sigma |\nabla\psi_0|^2 -\tau \psi_2 -\psi_1 + \tau c^2 \Delta \psi_0 + \tilde f(0), & \qquad \textrm{on } \Omega,  \\
         \xi_0&= 0 &  \qquad \textrm{on } \partial\Omega, 
    \end{array}   
    \right.
   \end{equation}
    Note that due to the regularity of the initial data and source term~\eqref{Main_system}, the right-hand side is in $H_0^1(\Omega)$. Thus, there exists a unique $\xi_0 \in \Honethree$ which solves \eqref{inital_cond_xi_0}; see, e.g., \cite[Theorem 2.5.1.1]{grisvard2011elliptic}.

	Then $\xi$ verifies the following equation
	\begin{equation}\label{eq:xi}
		\tau \xi_{ttt} + \xi_{tt} -c^2\Delta \xi-\tau c^2 \Delta \xi_t-\delta   \genk\Lconv \Delta (\tau \xi_{tt}+ \xi_t) =\sigma \nabla |\xi_{t}|^2 + \tilde f.
\end{equation}
This equation is obtained by integrating equation \eqref{eq:first_eq_genk} in time on $(0,t)$, keeping in mind \eqref{xi_Definition} and using the following relations:
\begin{equation}
	\begin{aligned}
		&\xi_t =  \psi = \intt \psit\ds + \psi_0  ,\qquad 
		\xi_{tt} =  \psit = \intt \psitt\ds +\psi_1,
        \\ & 		\xi_{ttt} =  \psitt = \intt \psittt\ds  + \psi_2, \qquad \genk \Lconv (\tau\xi_{tt} + \xi_t) =  \genk \Lconv (\tau\psi_t+\psi) = \intt \genk \Lconv (\tau\psitt+\psit) \ds,\qquad  \\&
		 \sigma|\nabla \xi_{t}|^2 = \sigma|\nabla\psi|^2 =\intt 2\sigma\nabla \psit\cdot \nabla \psi \ds + \sigma |\nabla \psi_0|^2.
	\end{aligned}
\end{equation}
We also used the  initial condition on $\xi_0$ given in \eqref{inital_cond_xi_0} to eliminate the time-independent residual terms that results from integrating \eqref{eq:first_eq_genk} with respect to time. 

To simplify the study of the equation of $\xi$, we reformulate \eqref{eq:xi} into the following system of equations:
\begin{equation}
	\left\{
	\begin{array}{ll}
		\xi_{t}=\psi,\vspace{0.2cm} &  \\
        z = \tau \xi_t + \xi, \vspace{0.2cm} &  \\
		z_{tt} - c^2 \D z - \delta \genk \Lconv \D z_t= \sigma |\nabla \psi|^2 +\tilde f , &
	\end{array}   
	\right.  \label{System_New_Integrated}
\end{equation}

Exploiting the similarity in structure between \eqref{System_New} and \eqref{System_New_Integrated}, we give a first estimate for the functional $E_2$, which we recall is defined as  
\begin{equation}  
    E_2(t)=\frac{1}{2}\int_{\Omega}\left( |\D(\psi+\tau \psit)|^{2}+c^2| \nabla \D (\xi+\tau \psi)|^{2}\right) \dx,
	\label{E1_analogous_def}
\end{equation}
We then have the following estimate. 
\begin{lemma}\label{lemma_E1L_identitty}
	We have  for all $t\geq 0$, the identity
	\begin{equation}
		E_2(t) + \delta  \int_0^t   \int_{\Omega} \genk\Lconv \nabla \D(\tau  \psit+\psi) \cdot  \nabla \D (\tau \psit+ \psi)\dx\ds \leq E_2(0)+\mathrm{R}_2,  \label{Energy_LOW_Indentity}
	\end{equation}
	where
	\begin{equation}
		\mathrm{R}_2:= -\int_0^t\int_{\Omega} \sigma \nabla |\nabla \psi|^2 \cdot
		\nabla \D \left( \psi+\tau \psit\right) \dx\ds + \int_0^t\int_{\Omega}  \D \tilde f \cdot \D \left( \psi+\tau \psit\right) \dx\ds.
	\end{equation} 
\end{lemma}
\begin{proof}
	Multiplying the third equation of \eqref{System_New_Integrated} by $ \D^2( \psi+\tau \psit) =  \D^2 z_t$ and integrating by parts
	over $\Omega$, we obtain
	\begin{eqnarray}
		\frac{1}{2}\ddt\int_{\Omega}|\D (\psi+\tau \psit)|^{2}dx + c^2 \int_{
			\Omega} \nabla \D( \xi + \tau \psi) \cdot \nabla \D (\psi + \tau \psit) \notag \\
            +\delta\int_{\Omega} \genk\Lconv \nabla \D (\psi + \tau \psit) \cdot \nabla \D (\psi+\tau \psit) \dx\notag \\
		=-\int_{\Omega} \sigma \nabla |\nabla \psi|^2 \cdot
		\nabla \D \left( \psi+\tau \psit\right) \dx +\int_{\Omega}  \D \tilde f \cdot \D \left( \psi+\tau \psit\right) \dx,  \label{w_LOW_estimate}
	\end{eqnarray}
     where we exploited the \emph{a priori} information on the vanishing boundary value of $\psi,$ $\D \psi,$ $\psi_t,$ and $ \D \psit$; see solution space~\eqref{def:space_of_solutions}. Due to the choice of data, we know that $\D \tilde f|_{\partial\Omega} =0$. Notice moreover that also $ \sigma |\nabla \psi|^2\Big|_{\partial \Omega} =0= \D \xi|_{\partial{\Omega}} $, which is due to the fact that:
    \begin{equation}
    \sigma |\nabla \psi(t)|^2\Big|_{\partial \Omega} = \int_0^t 2 \sigma \nabla \psi\cdot \nabla \psit \Big|_{\partial \Omega} \ds + \sigma  |\nabla \psi_0|^2\Big|_{\partial \Omega} =0
    \end{equation}
    due to the initial data choice and the space of solutions~\eqref{def:space_of_solutions}. The same reasoning is used for $\D \xi$.

	Exploiting the first equation in \eqref{System_New_Integrated} (adding $\tau \psit$ to both sides) and testing with $ - c^2 \D^3 ( \psit+\tau \psi)$, we have 
	\begin{equation}
		\begin{aligned}
			&\frac{1}{2}\ddt\int_{\Omega}c^2|\nabla \D (\xi+\tau \psi)|^{2}dx = c^2 \int_{
			\Omega} \nabla \D( \xi + \tau \psi) \cdot \nabla \D (\psi + \tau \psit) .  \label{v_ps_derivative}
		\end{aligned}
	\end{equation}
	Summing up \eqref{w_LOW_estimate} and \eqref{v_ps_derivative} and integrating with respect to $t$, we get obtain the desired estimate.
    
\end{proof}

Now, collecting the inequalities \eqref{Energy_Indentity} and \eqref{Energy_LOW_Indentity}, we show that although we are in the critical regime, the convolution term helps to gain a dissipative term that is crucial in controlling the nonlinearity and closing the energy estimate. This result is contained in the following lemma.

\begin{lemma}\label{lemma:L2_control_Deltaz}
We have for all $t\geq 0$
\begin{equation}\label{full_energy_ineq}
    E_1(t) + E_2(t) + \frac{\delta }{\eta} \intt \|\nabla \D (\tau \psit + \psi) \|_{L^2(\Omega)}^2 \ds \leq E_1(0) + E_2(0) + \mathrm{R}_1 + \mathrm{R}_2.
\end{equation}
\end{lemma}
\begin{proof}
Taking advantage of Lemma~\ref{lemma_L2_control}, we obtain 
\begin{equation}\label{L2_control_Dv}
    \begin{aligned}
	\int_0^{t} (\genk\Lconv \nabla \D (\tau \psitt + \psit), \nabla \D (\tau \psitt + \psit)) +  (\genk\Lconv \nabla \D (\tau \psit + \psi),  \nabla \D (\tau \psit + \psi)) \ds \\\geq \frac1\eta \intt\|\nabla \D(\tau \psit + \psi)\|^2_{L^2(\Omega)}\ds,
    \end{aligned}
\end{equation}
for some $\eta>0$.
  Hence,  summing up  \eqref{Energy_Indentity} and \eqref{Energy_LOW_Indentity} yields the desired control of the term $\|\tau \nabla \D \psit + \nabla\D \psi\|_{L^2_t L^2(\Omega)}$.  
\end{proof}

We show through the following lemma that the $L^2$-$L^2$ control provided by Lemma~\ref{lemma:L2_control_Deltaz} for the quantity 
$$ \intt \| \nabla \D(\tau  \psit + \psi) \|_{L^2(\Omega)}^2 \ds $$
can be used to extract suitable control for the individual component $\nabla \D\psi$ and $\nabla \D \psit$ in the spirit of \cite[Lemma 3.5]{lasiecka2017global}.
\begin{lemma} \label{lemma:Dissipation_control} Let $t\geq 0$, then we have the following bounds 
\begin{equation}
       \|\nabla \D \psi\|^2_{L^p(0,t;L^2(\Omega))}\lesssim  \intt \| \nabla \D (\tau \psit + \psi)\|_{L^2(\Omega)}^2\ds + \|\nabla \D \psi_0\|^2_{L^2(\Omega)} ,
\end{equation}
for all $p \in [2,\infty]$ and 
\begin{equation}
       \|\nabla \D \psit\|^2_{L^2(0,t;L^2(\Omega))}\lesssim  \intt \| \nabla\D (\tau \psit +\psi)\|_{L^2(\Omega)}^2\ds + \intt \|\nabla\D \psi\|^2_{L^2(\Omega)} \ds.
\end{equation}
The hidden constants are independent of time $t$.
\end{lemma}
\begin{proof}
We start with the relation 
\begin{equation}\label{eq:start_bootstrap_L2}
\tau \nabla\D \psit + \nabla \D \psi = \nabla \D w,
\end{equation}
then testing with $\nabla \D \psi$ and integration over $\Omega$ yields that
\begin{equation}
\frac\tau2 \ddt \|\nabla\D \psi\|^2_{L^2(\Omega)} + \|\nabla\D \psi\|^2_{L^2(\Omega)} \leq  \|\nabla\D w\|_{L^2(\Omega)} \|\nabla\D\psi\|_{L^2(\Omega)}.
\end{equation}
Integrating over $(0,\nu)$, and taking the supremum of $\nu \in (0,t)$, we obtain 
\begin{equation}
      \sup_{\nu \in (0,t)} \|\nabla\D \psi(\nu)\|^2_{L^2(\Omega)} + \intt \|\nabla\D \psi\|^2_{L^2(\Omega)} \ds \lesssim  \intt \|\nabla\D w\|_{L^2(\Omega)}^2\ds + \|\nabla\D \psi_0\|^2_{L^2(\Omega)} ,
\end{equation}
where the hidden constant only depends on the parameters of the problem but not on time $t$. 
By interpolation of $L^p$ spaces, we infer that 
\begin{equation}
       \|\nabla\D \psi\|^2_{L^p(0,t;L^2(\Omega))}\lesssim  \intt \|\nabla\D w\|_{L^2(\Omega)}^2\ds + \|\nabla\D \psi_0\|^2_{L^2(\Omega)} ,
\end{equation}
for $p \in [2,\infty]$.

For the second bound, we rely on the triangle inequality 
\begin{equation}
    \begin{aligned}
       \|\tau \nabla\D \psit\|_{L^2(0,t;L^2(\Omega))} \leq &\|\nabla\D (\tau\psi+\psi_t)\|_{L^2(0,t;L^2(\Omega))} + \|\nabla\D \psi\|_{L^2(0,t;L^2(\Omega))} \\
       \leq & \|\nabla\D (\tau\psi+\psi_t)\|_{L^2(0,t;L^2(\Omega))} + \|\nabla\D \psi_0\|^2_{L^2(\Omega)}.
    \end{aligned}
\end{equation}
\end{proof}

\begin{proposition}\label{lem::Energy_Diss_1D} Let $\Omega$ be a bounded Lipschitz domain. Then, for all time $t>0$, it holds that
\begin{equation}\label{ineq::Energy_Diss_1D}
				\mathbf{E}(t) + \mathbf{D}(t) \leq C_* \left[ \mathbf{E}(0) +\|\nabla \D \psi_0\|^4+ \|\tilde f\|_{W^{1,1}(\R^+; H^2(\Omega))}^2  + \mathbf{D}^2(t)\right],
			\end{equation}
          where $\mathbf{E}(t)$ and $\mathbf{D}(t)$ are defined in \eqref{Total_Energy} and \eqref{D_Term} respectively. Here,   the constant $C_*\geq1$ is independent of time $t$.
	\end{proposition}
   \begin{proof}
   It suffices to estimate the terms $\mathrm{R}_1$ and $\mathrm{R}_2$ in \eqref{full_energy_ineq}, which we do as follows: for an arbitrary $\varepsilon>0$, there exist $C(\varepsilon)$ such that
   \begin{equation}\label{R_1_Estimate}
   \begin{aligned}
   \mathrm{R}_1 = & \int_0^t\int_{\Omega}2\sigma \D (\nabla \psi \cdot \nabla \psit) \cdot 
		\D \left(  \psit+\tau  \psitt\right) \dx\ds   + \int_0^t\int_{\Omega} \D f\, \cdot \D \left( \psit+\tau \psitt\right) \dx\ds\\
    \leq &\, C(\varepsilon) \|\D f\|_{L^1(0,t; L^2(\Omega))}^2 + 2 \varepsilon \|\D (\psit+\tau \psitt)\|_{L^\infty(0,t;L^2(\Omega))}^2\\&+C(\epsilon)  \intt \| \D (\nabla \psi \cdot \nabla \psit)\|_{L^2(\Omega)}^2 \ds.
   \end{aligned}
   \end{equation}

  \noindent We next provide a bound on 
   \begin{equation}\label{L2estimate:nonlinearity}
   \begin{aligned}
   \| \D (\nabla \psi &\cdot \nabla \psit)\|_{L^2(\Omega)} \\ & \lesssim \|\nabla \psi\|_{H^2(\Omega)} \|\nabla \psit\|_{L^\infty(\Omega)} +  \|\nabla \psi\|_{W^{1,4}(\Omega)} \|\nabla \psit\|_{W^{1,4}(\Omega)} +   \|\nabla \psit\|_{H^2(\Omega)} \|\nabla \psi\|  _{L^\infty(\Omega)}   \\
   & \lesssim \|\nabla \D \psit\|_{L^2(\Omega)} \|\nabla \D \psi\|_{L^2(\Omega)},
   \end{aligned}
    \end{equation}
    where in the last line we used  the Sobelev embeddings $H^2 (\Omega) \hookrightarrow L^\infty(\Omega)\cap W^{1,4} (\Omega)$ as well as the elliptic estimate~\eqref{high_elliptic_reg_est} 
    \begin{equation}  \label{eq:elliptic_reg}
        \|\psi\|_{H^3(\Omega)} \lesssim \|\nabla \D \psi\|_{L^2},  
    \end{equation}
    since $\D\psi|_{\partial\Omega} =0$ due to the definition of the solution space in \eqref{def:space_of_solutions}.

    Thus, going back to \eqref{R_1_Estimate}, we obtain
    \begin{equation}
   \begin{aligned}
   \mathrm{R}_1 
    \leq &\, C(\varepsilon) \|\D f\|_{L^1(0,t; L^2(\Omega))}^2 + 2 \varepsilon \|\D (\psit+\tau \psitt)\|_{L^\infty(0,t;L^2(\Omega))}^2\\&+ C(\varepsilon)\mathbf{D}^2(t) + C(\varepsilon)\|\nabla \D \psi_0\|^4,
   \end{aligned}
   \end{equation}
   where we have used Lemma~\ref{lemma:L2_control_Deltaz}.
  
We now turn our attention to the term 
$$\mathrm{R}_2 = -\int_0^t\int_{\Omega} \sigma \nabla |\nabla \psi|^2 \cdot
		\nabla \D \left( \psi+\tau \psit\right) \dx\ds + \int_0^t\int_{\Omega}  \D \tilde f \cdot \D \left( \psi+\tau \psit\right) \dx\ds,$$ 
        whose terms can be estimated as follows:
        \begin{equation}
            \begin{aligned}
                \Big|\int_0^t\int_{\Omega} \sigma & \nabla |\nabla \psi|^2 \cdot
		\nabla \D \left( \psi+\tau \psit\right) \dx\ds\Big| \\\lesssim & \varepsilon \|\nabla \D \left( \psi+\tau \psit\right)\|_{L^\infty(0,t;L^2(\Omega))}^2  + C(\varepsilon) \left\||\nabla \psi|^2\right\|_{L^1(0,t;H^1(\Omega))}^2\\\lesssim & \varepsilon \|\nabla \D \left( \psi+\tau \psit\right)\|_{L^\infty(0,t;L^2(\Omega))}^2  + C(\varepsilon) \left\||\nabla \psi|^2\right\|_{L^1(0,t;H^2(\Omega))}^2
         \\ \lesssim & \varepsilon \|\nabla \D \left( \psi+\tau \psit\right)\|_{L^\infty(0,t;L^2(\Omega))}^2  + C(\varepsilon) \left\| \nabla \psi\right\|_{L^2(0,t;H^2(\Omega))}^4
        \\ \lesssim & \varepsilon \|\nabla \D \left( \psi+\tau \psit\right)\|_{L^\infty(0,t;L^2(\Omega))}^2  + C(\varepsilon) \left\| \psi\right\|_{L^2(0,t;H^3(\Omega))}^4,
            \end{aligned}
        \end{equation} 
    where in the last line we used that $H^2(\Omega)$ is an algebra for $\Omega \in \R^d$ with $d\leq3$. Using Lemma~\ref{lemma:Dissipation_control}, we furthermore estimate:
    $$\left\| \psi\right\|_{L^2(0,t;H^3(\Omega))}^4 \lesssim \mathbf{D}^2(t) + \|\nabla \D \psi_0\|^4$$
    The estimate for the source term $\tilde f$ is treated similarly to above, this yields the inequality:  \begin{equation}
   \begin{aligned}
   \mathrm{R}_2
    \leq & C(\varepsilon)\|\D \tilde f\|_{L^1(0,t; L^2(\Omega))}^2 + 2 \varepsilon \mathbf{E}(t)
    + C(\varepsilon) \mathbf{D}^2(t) + C(\varepsilon)\|\nabla \D \psi_0\|^4.
   \end{aligned}
   \end{equation}
   Thus, by choosing $\varepsilon$ small enough, we obtain \eqref{ineq::Energy_Diss_1D}. 
\end{proof}
\begin{proof}[Proof of Theorem \ref{thm:main_thm}]
Let $t>0$ and define $\mathbf{Y}(t):= \mathbf{E}(t) + \mathbf{D}(t) $. Hence, using \eqref{ineq::Energy_Diss_1D}, we obtain   
\begin{equation}
\mathbf{Y}(t) \lesssim  \mathbf{Y}(0) + \|\tilde f\|^2_{W^{1,1}(\R^+; H^2(\Omega))} + \|\nabla \D \psi_0\|^4 + \mathbf{Y}(t)^2 .
\end{equation}
Applying Lemma ~\ref{Lemma_Stauss}, we deduce that $\mathbf{Y}(t)$ is uniformly bounded provided that $\|\psi_2\|_{L^2}+\|\tilde f\|_{W^{1,1}} $ is sufficiently small. 
 This complete the  proof of Theorem~\ref{thm:main_thm}.  We refer the reader to, e.g., \cite[p. 8]{racke2021global} and \cite[p. 15]{ji2022optimal} for similar arguments, sometimes based on a bootstrap argument by Tao~\cite[Proposition 1.21]{tao2006nonlinear} comparable to Lemma~\ref{Lemma_Stauss}.
\end{proof}

\section{conclusion}
In this work, and even though we are in the critical regime in terms of the coefficients, we have shown that fractional damping is sufficient to ensure global existence for the nonlinear JMGT equation with quadratic gradient nonlinearity when the problem data are small. We also ensure that the solution's energy remains small as time goes to infinity. A natural future question will be to study the precise asymptotic behavior as time goes to infinity. In particular, we would like to know under which conditions on the problem data and on the memory kernel, the solution's energy decays, and to determine the corresponding decay rate.  

 Future work will also be tasked with proving the global existence of the fractional JMGT equation with Westervelt nonlinearity (i.e., $\psitt \psit$). In the present framework, we do not obtain sufficient control ($L^2$-control specifically) on the quantity $\psitt$ to argue global existence for Westervelt-type with fractional damping. Notice that this differs from the  strong damping case ($\genk = \delta_0$, $\delta>0$), where the full energy norm of the linearized equation can be shown to decay; see, e.g., \cite{kaltenbacher2011wellposedness}. Thus, one likely needs additional assumptions to be made on the memory kernel to achieve this result.

\section*{Acknowledgments}
	The work of the first author was supported by the Czech Sciences Foundation (GA\v CR), Grant Agreement 24-11034S. The Institute of Mathematics of the Academy of Sciences of the Czech Republic is supported by RVO:67985840. 
 
\appendix

\section{Local well-posedness result}\label{App:Appendix}
We present in this appendix, the results concerning the local well-posed of \eqref{System_New}--\eqref{Boundary_Condition_3}. The proof is largely inspired from that of \cite[Theorem 3.1]{meliani2024well} with changes made to accomodate the quadratic gradient nonlinearity. 
\begin{theorem}\label{prop::local_well_posedness_1d}  Let $\Omega$ be a $C^{2,1}-$regular domain. Let $\tau>0$, $\delta \geq 0$, and Assumption~\ref{assu:kernel} on $\genk$ hold. Assume that  $(\psi_0, \psi_1,\psi_2) \in \Honethree \times \Honethree \times \Honetwo$ and $f\in L^1(\R^+;\Honetwo)$
	such that 
	\begin{equation}\label{N_0}
\|\psi_0\|_{H^3(\Omega)}^2 + \|\psi_1\|_{H^3(\Omega)}^2 + \|\psi_2\|_{H^2(\Omega)}^2 + \|f\|_{L^1(\R^+;H^2(\Omega))}^2  \leq N_0,
\end{equation}
	for some $N_0>0$. Then, there exits $T_*>0$ that only depends on $N_0$ such that  the system \eqref{System_New}--\eqref{Boundary_Condition_3} admits a unique solution $\psi \in \mathcal{U}(0,t)$, at least up to $T_*$. Furthermore, for all $t \in [0,T_*]$, the following estimate holds:
 \begin{equation}
\|\psi\|_{\mathcal{U}(0,t)} \leq C(T_*) N_0,
\end{equation}
where the constant $C(T_*)$ depends on $T_*$. 
\end{theorem}
\begin{proof}
	The proof of Theorem \ref{prop::local_well_posedness_1d} uses a Galerkin approximation combined with energy estimates and compactness arguments. 
    We may approximate the solution by
	\begin{equation}\label{aansatz:psi}
		\begin{aligned}
			\psi^{n}(x,t) =\sum_{i=1}^n \xi^n_i(t)v_i(x),
		\end{aligned}
	\end{equation}
	where $\{v_i\}_{i=1}^\infty$ are the complete set of smooth eigenfunctions of the Dirichlet-Laplacian operator; see, e.g., \cite{kaltenbacher2022parabolic} for similar Galerkin approximation arguments. The natural number $n$ is the Galerkin approximation level. The existence of semi-discrete approximate solution is ensured by Volterra integral equations theory along the lines of \cite[Theorem 3.1]{meliani2024well}.

    We focus here on establishing the necessary energy estimate which are of higher order in space compared to the aforementioned \cite{meliani2024well} due to the presence of the quadratic gradient nonlinearity. We also discuss how to  pass  to the limit in the Galerkin procedure.

\subsection*{Step (i): Uniform estimates} 
We establish here \emph{a priori} estimates of the solutions that are uniform with respect to the approximation level $n$.
In what follows, we omit the superscript $n$ to simplify notation. Furthermore, let us recall the energy,
\begin{equation}
	E_{1}(t)=\frac{1}{2}\int_{\Omega}\left( |\D(\psit+\tau \psitt)|^{2}+c^2|\nabla \D  (\psi+\tau \psit)|^{2}\right) \dx
\end{equation}
and introduce the functional $F$ which is simply defined as 
\begin{equation}
    F(t) := \sup_{\mu \in (0,t)} E_1(\mu).
\end{equation}

Let $T>0$ be an arbitrary final time, which will be refined in Step (iii) of this proof to establish a lower bound on the final time of existence. 
Let $t\in[0,T]$. We begin by testing \eqref{System_New}  with $ \D^2 (\tau \psitt +\psit) $ and integrating over $\Omega$ and $(0,t)$, to obtain 
\begin{equation}\label{eq:first_estimate_linearized}
    \begin{multlined}
		E_1(t) + \delta  \int_{\Omega} \genk\Lconv \nabla \D(\tau  \psitt +\psit) \cdot  \nabla \D (\tau \psitt + \psit)\dx  \\
		=  E_1(0) + \intt \int_{\Omega}2\sigma (\nabla \psi \cdot \nabla \psit)  
		\D^2 \left(  \psit+\tau  \psitt\right) +  f \D^2 \left( \psit+\tau \psitt\right) \dxs. 
        \end{multlined}
	\end{equation}

We show, in what follows, how to control each of the arising terms in \eqref{eq:first_estimate_linearized}. First, using Assumption~\eqref{assu:kernel} on the kernel, which ensures its positivity in the sense of inequality~\eqref{ineq:positive_kernel}, we ensure:
\begin{equation}\label{Est_Conv}
\delta  \int_{\Omega} \genk\Lconv \nabla \D(\tau  \psitt +\psit) \cdot  \nabla \D (\tau \psitt + \psit)\dx \geq 0.
\end{equation}

Next, using the fact that $\psi$ are combinations of the eigenfunction of the Dirichlet-Laplacian operator (see \eqref{aansatz:psi}), we infer that $\psi|_{\partial \Omega}=\Delta\psi|_{\partial \Omega} = 0.$ This implies that 
\begin{equation}\label{eq:boundary_nonlin}
2 \sigma \nabla \psi \cdot \nabla \psi_t |_{\partial\Om}=   - \left[\tau \psittt+\psitt
-c^2 \Delta \psi - \tau c^2 \Delta\psit
- \delta \genk \Lconv\Delta (\tau\psitt+ \psit) - f \right]|_{\partial\Om} =0,
\end{equation}
where we used that $f|_{\partial \Om} =0.$

Thus, we can perform integration by parts and use H\"older's inequality to infer that 
\begin{equation}\label{eq:psit_psitt_product}
	\begin{aligned}
	&\hphantom{=\ } \left| \intt  (\nabla \psi \cdot \nabla \psit,\D^2 (\tau\psitt+ \psit))_{L^2(\Omega)} \ds \right|  \\ & = \left| \intt  (\D (\nabla \psi \cdot \nabla \psit),\D (\tau\psitt+ \psit))_{L^2(\Omega)} \ds \right|  \\ &\leq  \intt \| \D (\nabla \psi \cdot \nabla \psit) \|_{L^2(\Om)} \| \D (\tau\psitt+ \psit)\|_{L^2(\Om)}\ds.
	\end{aligned}
\end{equation}

\noindent Similarly to \eqref{L2estimate:nonlinearity}, we bound:
\begin{equation}\label{some_inequalities}
    \| \D (\nabla \psi \cdot \nabla \psit)\|_{L^2(\Omega)}  \lesssim \|\nabla \D \psit\|_{L^2(\Omega)} \|\nabla \D \psi\|_{L^2(\Omega)}.
\end{equation}

Next, we notice that, similarly to \eqref{eq:start_bootstrap_L2}, we can set:
\begin{equation}
\tau \nabla\D \psit + \nabla \D \psi = \nabla \D z,
\end{equation}
and write the convolution solution representation for the first order ODE above for $t\geq 0 $
\begin{equation}\label{eq:ODE_rep_conv}
    \nabla \D \psi: t\mapsto  \frac1\tau e^{-\,\cdot\,/\tau} \Lconv \nabla \D z  + e^{-t/\tau} \nabla \D\psi_0.
\end{equation}
From \eqref{eq:ODE_rep_conv}, we infer that 
\begin{equation}
\begin{aligned}\label{eq:bound_psi}
        \|\nabla \D \psi(t)\|_{L^2(\Omega)} &\leq \sup_{\nu \in(0,t)}  \|\nabla \D \psi(\nu)\|_{L^2(\Omega)}  
    \\& 
    \leq \left\|\frac1\tau e^{-\,\cdot\,/\tau} \right\|_{L^1(0,t)} \sup_{\nu \in(0,t)} \|\nabla \D z\|_{L^2(\Omega)} + \|\nabla \D \psi_0\|_{L^2(\Om)}
    \\
    & \leq F^{1/2} + \|\nabla\D\psi_0\|_{L^2(\Omega)}.
\end{aligned}
\end{equation}
In the last line, we have used that \[\left\|\frac1\tau e^{-\,\cdot\,/\tau} \right\|_{L^1(0,t)} \leq \left\|\frac1\tau e^{-\,\cdot\,/\tau} \right\|_{L^1(0,\infty)} = 1\] 
and that 
\[\sup_{\nu \in(0,t)} \|\nabla \D z\|_{L^2(\Omega)} \leq F^{1/2}.\]

Using the triangle inequality on $\tau \nabla\D \psit = - \nabla \D \psi + \nabla \D z,$ and combing it with estimate~\eqref{eq:bound_psi}, we obtain the estimate:
\begin{equation} \label{eq:bound_psit}
\|\nabla \D \psit(t)\|_{L^2(\Omega)}  \lesssim F^{1/2} + \|\nabla\D\psi_0\|_{L^2(\Omega)},
\end{equation}
with the hidden constant being independent of time. 

 Using \eqref{some_inequalities} together with \eqref{eq:bound_psi}--\eqref{eq:bound_psit}, we can bound \eqref{eq:psit_psitt_product} as follows
\begin{equation}
	\begin{aligned}
		\Big| \intt  (\nabla \psi \cdot \nabla & \psit,\D^2 (\tau\psitt+ \psit))_{L^2(\Omega)} \ds \Big| \\ &\lesssim  \intt F^{3/2}(s)\ds + \|\nabla \D\psi^2\|_{L^2(\Om)} \intt F(s) \ds+ \|\nabla \D\psi_0\|^2_{L^2(\Om)} \intt F^{1/2} \ds.
	\end{aligned}
\end{equation}

For the source term, we integrate by parts and use H\"older's and Young's inequalities as follows:
\begin{equation}
    \intt \intO f \D^2 \left( \psit+\tau \psitt\right) \dxs \leq C(\varepsilon) \|\D f\|^2_{L^1(\R^+;\,L^2(\Om))} + \varepsilon F(t), 
\end{equation}
for arbitrary $\varepsilon>0$.

Inserting the above estimate into \eqref{eq:first_estimate_linearized}, taking the supremum over $(0,t)$ and taking $\varepsilon$ small enough to be absorbed by the left-hand side, we obtain
\begin{equation}\label{estimate:nonlinear}
    \begin{multlined}
		F(t) \lesssim   E_1(0) + \|\D f\|^2_{L^1(\R^+;\,L^2(\Om))}+ \intt F^{3/2}(s)\ds + \|\nabla \D\psi_0\|_{L^2(\Om)} \intt F(s) \ds \\+ \|\nabla \D\psi_0\|^2_{L^2(\Om)} \intt F^{1/2} \ds. 
        \end{multlined}
	\end{equation}

\subsection*{Applying Gronwall's lemma}
The functional  $F$ is bounded up to a time $T_*=T_*(N_0)$ which depends on the size of the initial data. To see this, notice that by Gronwall's lemma; see, e.g., \cite[p. 47]{john1990nonlinear},  estimate \eqref{estimate:nonlinear} ensures that  $$F(t)\leq z(t),$$
where $z(t)$ is the solution of	 \begin{equation}\label{eqn:bound_equation}
			z'=c_0 z^\frac12 + c_1 z(s)+ c_2 z(s)^\frac32,
		\end{equation}
where $z(0) = C_* \left[E_1(0) + \|\D f\|^2_{L^1(\R^+;\,L^2(\Om))}\right]$, with $C_*$ being the hidden constant in \eqref{estimate:nonlinear}. The constants $c_{1,2,3}$ correspond to the coefficient in \eqref{estimate:nonlinear} (multiplied by the hidden constant $C_*$). Classical ODE theory guarantees the existence of at least a local-in-time solution to \eqref{eqn:bound_equation}. Note that \eqref{eqn:bound_equation} only depends on the size of the data of the problem and thus the final time $T_*$ only depends on $N_0$.

\subsection*{Step (ii): Passing to the limit}
From the previous estimates, we know that the sequence of approximate solutions $\psi^n$ (where we stress again the dependence on the Galerkin approximation level $n$) stays bounded uniformly with respect to $n$. In particular, we have the following weak(-$\star$) convergence of a subsequence (which we do not relabel):  
 \begin{alignat}{2}. 
	\psi^{n} &\stackrel{}{\relbar\joinrel\rightharpoonup} \psi \quad &&\textrm{weakly-$*$ in } L^\infty(0,T_*; \Honethree),\\
	\psit^{n} &\stackrel{}{\relbar\joinrel\rightharpoonup} \psit \quad &&\textrm{weakly-$*$ in }L^\infty(0,T_*; \Honethree), \\
	\psitt^{n} &\stackrel{}{\relbar\joinrel\rightharpoonup} \psitt \quad &&\textrm{weakly-$*$ in }L^\infty(0,T_*; \Honetwo).
\end{alignat}
Using the well-known Aubin--Lions--Simon lemma~\cite[Thoerem 5]{simon1986compact}, we know that there exists a strongly convergent subsequence (again not relabeled) in the following sense
 \begin{alignat}{2}\label{eq:strong_convergence}
	\psi^{n} &\longrightarrow \psi \quad &&\textrm{strongly in } C([0,T_*]; \Honetwo),\\
	\psit^{n} &\longrightarrow \psit \quad &&\textrm{strongly in }C([0,T_*]; \Honetwo).
\end{alignat}

Additionally, since $\nabla \psi^n\cdot \nabla \psit^n$ is bounded in $L^\infty(0,T_*;\Honetwo)$ (see estimate~\eqref{L2estimate:nonlinearity}), we infer that 
 \begin{alignat}{2}
	\nabla\psit^n\cdot\nabla \psi^n &\stackrel{}{\relbar\joinrel\rightharpoonup} y \quad &&\textrm{weakly-$*$ in }L^\infty(0,T_*; H_0^1(\Omega)).
\end{alignat}
To show that $y = \nabla\psit\cdot\nabla \psi$ weakly, we use the fact that $\nabla\psit^n\cdot\nabla \psi^n$ is the product of a weakly convergent sequence in $L^\infty(0,T_*;H^2(\Omega))$ and a strongly convergent sequence in $C([0,T_*]; H^1(\Omega))$. The fact that $\nabla\psi\cdot\nabla\psit |_{\partial(\Omega)} =0 $ is inherited form \eqref{eq:boundary_nonlin} by passing to the limit in the $H^1(\Omega)$ norm.

We mention that since $\genk\in L^1_\textup{loc}(\R^+)$, we also obtain the weak convergence:
 \begin{alignat}{2}
	\genk\Lconv \D\psitt &\stackrel{}{\relbar\joinrel\rightharpoonup} y \quad &&\textrm{weakly-$*$ in }L^\infty(0,T_*; L^2(\Omega)).
\end{alignat}
 
This is enough to show that the limit $\psi \in \mathcal{U}(0,T_*)$ solves the weak form 
\begin{equation} \label{eq:weak_form_local}
    \begin{multlined}
	-\int_0^{T_*}(\tau \psitt, v_t)_{L^2(\Omega)}\ds  
    +\int_0^{T_*} ((1+2k\psit)\psitt, v)_{L^2(\Omega)}\ds
	\\ + \int_0^{T_*} (c^2 \Delta \psi + \tau c^2 \Delta\psit
	+ \delta   \genk\Lconv \Delta \psitt,  v)_{L^2(\Omega)} \ds \\= \int_0^{T_*} (f,v)_{L^2(\Om)}\ds - \tau (\psi_2, v(0))_{L^2(\Omega)}\ds ,
        \end{multlined}
\end{equation}
for all $v \in H^1(0,T_*;L^2(\Om)) \cap L^2(0,T_*;H_0^1)$ such that $v(T_*)=0$.
Attainment of initial data of the constructed solution is guaranteed by the combination of the strong convergence of $\psi,\, \psit$ and the convergence of the $L^2(\Om)$ projections of the initial data. The attainment of the third initial datum is guaranteed weakly through the weak form~\eqref{eq:weak_form_local}.

\subsection*{Bootstrap estimate for \texorpdfstring{$\psittt$}{psittt}}
Since 
\[\tau \psittt  = - \psitt +  c^2\Delta \psi + \tau c^2 \Delta \psit + \delta \genk \Lconv\Delta (\tau\psitt+ \psit)- \sigma \nabla \psi \cdot \nabla \psi_t - f,\]
the triangle inequality yields a bound on the norm of $\|\psittt\|_{L^1(0,T_*; H_0^1(\Omega))}$.

Uniqueness of solutions, follows from standard argument, see, e.g., \cite{meliani2024well,meliani2023unified}. Its details are omitted.

\end{proof}



\begin{thebibliography}{10}
	
	\bibitem{brezis1980nonlinear}
	{\sc H.~Brezis and T.~Gallouet}, {\em Nonlinear schr{\"o}dinger evolution
		equations}, Nonlinear Analysis, 4 (1980), pp.~677--681.
	
	\bibitem{brezis1980note}
	{\sc H.~Br{e}zis and S.~Wainger}, {\em A note on limiting cases of {S}obolev
		embeddings and convolution inequalities}, Communications in Partial
	Differential Equations, 5 (1980), pp.~773--789.
	
	\bibitem{cannarsa2011integro}
	{\sc P.~Cannarsa and D.~Sforza}, {\em Integro-differential equations of
		hyperbolic type with positive definite kernels}, Journal of Differential
	Equations, 250 (2011), pp.~4289--4335.
	
	\bibitem{Caputo:1967}
	{\sc M.~Caputo}, {\em Linear models of dissipation whose {$Q$} is almost
		frequency independent -- {II}}, Geophys. J. Int., 13 (1967), pp.~529--539.
	
	\bibitem{Chen_Ikehata_2021}
	{\sc W.~Chen and R.~Ikehata}, {\em The {C}auchy problem for the
		{M}oore-{G}ibson-{T}hompson equation in the dissipative case}, J.
	Differential Equations, 292 (2021), pp.~176--219.
	
	\bibitem{chen2022influence}
	{\sc W.~Chen, Y.~Liu, A.~Palmieri, and X.~Qin}, {\em The influence of viscous
		dissipations on the nonlinear acoustic wave equation with second sound},
	arXiv preprint arXiv:2211.00944,  (2022).
	
	\bibitem{compte1997generalized}
	{\sc A.~Compte and R.~Metzler}, {\em The generalized {C}attaneo equation for
		the description of anomalous transport processes}, Journal of Physics A:
	Mathematical and General, 30 (1997), p.~7277.
	
	\bibitem{conejero2015chaotic}
	{\sc J.~A. Conejero, C.~Lizama, and F.~d.~A. R{\'o}denas~Escrib{\'a}}, {\em
		Chaotic behaviour of the solutions of the {M}oore--{G}ibson--{T}hompson
		equation}, Applied Mathematics \& Information Sciences, 9 (2015),
	pp.~2233--2238.
	
	\bibitem{conti2021Moore}
	{\sc M.~Conti, L.~Liverani, and V.~Pata}, {\em On the
		{M}oore-{G}ibson-{T}hompson equation with memory with nonconvex kernels},
	Indiana University Mathematics Journal,  (2021).
	
	\bibitem{Djrbashian:1966}
	{\sc M.~M. Dzhrbashyan}, {\em Integral {T}ransformations and {R}epresentation
		of {F}unctions in a {C}omplex {D}omain [in {R}ussian]}, Nauka, Moscow, 1966.
	
	\bibitem{fritz2022equivalence}
	{\sc M.~Fritz, U.~Khristenko, and B.~Wohlmuth}, {\em Equivalence between a
		time-fractional and an integer-order gradient flow: The memory effect
		reflected in the energy}, Advances in Nonlinear Analysis, 12 (2022),
	p.~20220262.
	
	\bibitem{grisvard2011elliptic}
	{\sc P.~Grisvard}, {\em Elliptic problems in nonsmooth domains}, SIAM, 2011.
	
	\bibitem{ji2022optimal}
	{\sc R.~Ji, L.~Yan, and J.~Wu}, {\em Optimal decay for the {3D} anisotropic
		{B}oussinesq equations near the hydrostatic balance}, Calculus of Variations
	and Partial Differential Equations, 61 (2022), p.~136.
	
	\bibitem{john1990nonlinear}
	{\sc F.~John}, {\em Nonlinear wave equations, formation of singularities},
	vol.~2, American Mathematical Soc., 1990.
	
	\bibitem{jordan2014second}
	{\sc P.~M. Jordan}, {\em Second-sound phenomena in inviscid, thermally relaxing
		gases}, Discrete \& Continuous Dynamical Systems-B, 19 (2014), p.~2189.
	
	\bibitem{Klaten_2015}
	{\sc B.~Kaltenbacher}, {\em Mathematics of nonlinear acoustics}, Evol. Equ.
	Control Theory, 4 (2015), pp.~447--491.
	
	\bibitem{kaltenbacher2011wellposedness}
	{\sc B.~Kaltenbacher, I.~Lasiecka, and R.~Marchand}, {\em Wellposedness and
		exponential decay rates for the {M}oore--{G}ibson--{T}hompson equation
		arising in high intensity ultrasound}, Control and Cybernetics, 40 (2011),
	pp.~971--988.
	
	\bibitem{kaltenbacher2012well}
	{\sc B.~Kaltenbacher, I.~Lasiecka, and M.~K. Pospieszalska}, {\em
		Well-posedness and exponential decay of the energy in the nonlinear
		{J}ordan--{M}oore--{G}ibson--{T}hompson equation arising in high intensity
		ultrasound}, Mathematical Models and Methods in Applied Sciences, 22 (2012),
	p.~1250035.
	
	\bibitem{kaltenbacher2024kuznetsov}
	{\sc B.~Kaltenbacher, M.~Meliani, and V.~Nikoli{\'c}}, {\em The {Kuznetsov and
			Blackstock} equations of nonlinear acoustics with nonlocal-in-time
		dissipation}, Applied Mathematics \& Optimization, 89 (2024), p.~63.
	
	\bibitem{kaltenbacher2024limiting}
	\leavevmode\vrule height 2pt depth -1.6pt width 23pt, {\em Limiting behavior of
		quasilinear wave equations with fractional-type dissipation}, Advanced
	Nonlinear Studies, 24 (2024), pp.~748--774.
	
	\bibitem{KaltenbacherNikolic}
	{\sc B.~Kaltenbacher and V.~Nikoli\'c}, {\em The
		{J}ordan--{M}oore--{G}ibson--{T}hompson equation: {W}ell-posedness with
		quadratic gradient nonlinearity and singular limit for vanishing relaxation
		time}, Mathematical Models and Methods in Applied Sciences, 29 (2019),
	pp.~2523--2556.
	
	\bibitem{kaltenbacher2022parabolic}
	\leavevmode\vrule height 2pt depth -1.6pt width 23pt, {\em Parabolic
		approximation of quasilinear wave equations with applications in nonlinear
		acoustics}, SIAM Journal on Mathematical Analysis, 54 (2022), pp.~1593--1622.
	
	\bibitem{kaltenbacher2022time}
	{\sc B.~Kaltenbacher and V.~Nikoli{\'c}}, {\em Time-fractional
		{M}oore--{G}ibson--{T}hompson equations}, Mathematical Models and Methods in
	Applied Sciences, 32 (2022), pp.~965--1013.
	
	\bibitem{kaltenbacher2024vanishing}
	\leavevmode\vrule height 2pt depth -1.6pt width 23pt, {\em The vanishing
		relaxation time behavior of multi-term nonlocal
		{Jordan--Moore--Gibson--Thompson} equations}, Nonlinear Analysis: Real World
	Applications, 76 (2024), p.~103991.
	
	\bibitem{kawashima1993global}
	{\sc S.~Kawashima}, {\em Global solutions to the equation of viscoelasticity
		with fading memory}, Journal of Differential Equations, 101 (1993),
	pp.~388--420.
	
	\bibitem{lasiecka2017global}
	{\sc I.~Lasiecka}, {\em Global solvability of {M}oore--{G}ibson--{T}hompson
		equation with memory arising in nonlinear acoustics}, Journal of Evolution
	Equations, 17 (2017), pp.~411--441.
	
	\bibitem{lions1965some}
	{\sc J.-L. Lions and W.~A. Strauss}, {\em Some non-linear evolution equations},
	Bulletin de la Soci{\'e}t{\'e} Math{\'e}matique de France, 93 (1965),
	pp.~43--96.
	
	\bibitem{marchand2012abstract}
	{\sc R.~Marchand, T.~McDevitt, and R.~Triggiani}, {\em An abstract semigroup
		approach to the third-order {M}oore--{G}ibson--{T}hompson partial
		differential equation arising in high-intensity ultrasound: structural
		decomposition, spectral analysis, exponential stability}, Mathematical
	Methods in the Applied Sciences, 35 (2012), pp.~1896--1929.
	
	\bibitem{meliani2023unified}
	{\sc M.~Meliani}, {\em A unified analysis framework for generalized fractional
		{M}oore--{G}ibson--{T}hompson equations: Well-posedness and singular limits},
	Fractional Calculus and Applied Analysis,  (2023).
	
	\bibitem{meliani2025energy}
	{\sc M.~Meliani and B.~Said-Houari}, {\em Energy decay of some multi-term
		nonlocal-in-time {Moore--Gibson--Thompson} equations}, Journal of
	Mathematical Analysis and Applications, 542 (2025), p.~128791.
	
	\bibitem{meliani2024well}
	\leavevmode\vrule height 2pt depth -1.6pt width 23pt, {\em Well-posedness and
		global extensibility criteria for time-fractionally damped
		{Jordan--Moore--Gibson--Thompson} equation}, Nonlinear Differential Equations
	and Applications NoDEA, 32 (2025), pp.~1--26.
	
	\bibitem{messaoudi2007global}
	{\sc S.~A. Messaoudi, B.~Said-Houari, and N.-e. Tatar}, {\em Global existence
		and asymptotic behavior for a fractional differential equation}, Applied
	Mathematics and Computation, 188 (2007), pp.~1955--1962.
	
	\bibitem{munoz1996decay}
	{\sc J.~E. Mu{\~n}oz~Rivera and E.~C. Lapa}, {\em Decay rates of solutions of
		an anisotropic inhomogeneous n-dimensional viscoelastic equation with
		polynomially decaying kernels}, Communications in Mathematical Physics, 177
	(1996), pp.~583--602.
	
	\bibitem{nikolic2024infty}
	{\sc V.~Nikoli{\'c} and M.~Winkler}, {\em ${L}^\infty$ blow-up in the
		{J}ordan--{M}oore--{G}ibson--{T}hompson equation}, Nonlinear Analysis, 247
	(2024), p.~113600.
	
	\bibitem{nohel1976frequency}
	{\sc J.~Nohel and D.~Shea}, {\em Frequency domain methods for {V}olterra
		equations}, Advances in Mathematics, 22 (1976), pp.~278--304.
	
	\bibitem{pellicer2019wellposedness}
	{\sc M.~Pellicer and B.~Said-Houari}, {\em Wellposedness and decay rates for
		the {C}auchy problem of the {M}oore--{G}ibson--{T}hompson equation arising in
		high intensity ultrasound}, Applied Mathematics \& Optimization, 80 (2019),
	pp.~447--478.
	
	\bibitem{P-SM-2019}
	{\sc M.~Pellicer and J.~Sol{\`a}-Morales.}, {\em Optimal scalar products in the
		{M}oore--{G}ibson--{T}hompson equation}, Evol. Eqs. and Control Theory, 8
	(2019), pp.~203--220.
	
	\bibitem{racke2021global}
	{\sc R.~Racke and B.~Said-Houari}, {\em Global well-posedness of the {C}auchy
		problem for the {3D} {J}ordan--{M}oore--{G}ibson--{T}hompson equation},
	Communications in Contemporary Mathematics, 23 (2021), p.~2050069.
	
	\bibitem{said2022global}
	{\sc B.~Said-Houari}, {\em Global existence for the
		{Jordan--Moore--Gibson--Thompson} equation in {B}esov spaces}, Journal of
	Evolution Equations, 22 (2022), p.~32.
	
	\bibitem{simon1986compact}
	{\sc J.~Simon}, {\em Compact sets in the space ${L_p(0, T; B)}$}, Annali di
	Matematica pura ed applicata, 146 (1986), pp.~65--96.
	
	\bibitem{Strauss_1968}
	{\sc W.~A. Strauss}, {\em Decay and asymptotics for $\square u= {F} (u)$}, J.
	Functional Analysis, 2 (1968), pp.~409--457.
	
	\bibitem{tao2006nonlinear}
	{\sc T.~Tao}, {\em Nonlinear dispersive equations: local and global analysis},
	no.~106, American Mathematical Soc., 2006.
	
	\bibitem{vergara2008lyapunov}
	{\sc V.~Vergara and R.~Zacher}, {\em Lyapunov functions and convergence to
		steady state for differential equations of fractional order}, Mathematische
	Zeitschrift, 259 (2008), pp.~287--309.
	
	\bibitem{vergara2015optimal}
	\leavevmode\vrule height 2pt depth -1.6pt width 23pt, {\em Optimal decay
		estimates for time-fractional and other nonlocal subdiffusion equations via
		energy methods}, SIAM Journal on Mathematical Analysis, 47 (2015),
	pp.~210--239.
	
\end{thebibliography}
\end{document}